\pdfoutput=1
\documentclass[fleqn,a4paper]{article}
\usepackage[T2A]{fontenc}
\usepackage[utf8]{inputenc}
\usepackage[english]{babel}
\usepackage{a4wide,enumerate}
\usepackage{amssymb,amsthm,mathtools,mathrsfs}
\usepackage{graphicx}
\usepackage{indentfirst}

\usepackage{ifpdf}

\ifpdf
\usepackage[bookmarks=false,%                  % Comment
pdfstartview=FitH,linkbordercolor={0.5 1 1},%  % these lines if
citebordercolor={0.5 1 0.5},unicode,%          % hyperrefs in PDF
pagebackref,hyperindex]{hyperref}%             % are not needed!
\else
\fi

\usepackage{cite}

\newcommand{\setA}{\mathscr{A}}
\newcommand{\setB}{\mathscr{B}}
\newcommand{\Hset}{\mathcal{H}}
\newcommand{\Lin}{\mathcal{L}}
\newcommand{\IRU}{\mathcal{U}}
\newcommand{\bbR}{\mathbb{R}}

\newtheorem{theorem}{Theorem}

\newtheorem{example}{Example}

\newtheorem{problem}{Problem}
\newtheorem{remark}{Remark}

\begin{document}
\date{}

\title{Constructive stability and stabilizability
of positive linear discrete-time switching systems}

\author{Victor Kozyakin\thanks{Kharkevich Institute for Information Transmission
Problems, Russian Academy of Sciences, Bolshoj Karetny lane 19, Moscow
127051, Russia, e-mail: kozyakin@iitp.ru\newline\indent~~Kotel'nikov
Institute of Radio-engineering and Electronics, Russian Academy of Sciences,
Mokhovaya 11-7, Moscow 125009, Russia}}

\maketitle

\begin{abstract}
We describe a new class of positive linear discrete-time switching systems
for which the problems of stability or stabilizability can be resolved
constructively. The systems constituting this class can be treated as a
natural generalization of systems with the so-called independently
switching state vector components. Distinctive feature of such systems is
that their components can be arbitrarily `re-connected' in parallel or in
series without loss of the `constructive resolvability' property for the
problems of stability or stabilizability of a system. It is shown also
that, for such systems, the individual positive trajectories with the
greatest or the lowest rate of convergence to the zero can be built
constructively.
\end{abstract}

\section{Introduction}
A linear discrete-time system
\begin{equation}\label{E:main}
x(n+1)=A(n)x(n),\quad x(n)\in\mathbb{R}^{N},
\end{equation}
is called \emph{switching} provided that the $(N\times N)$-matrices $A(n)$,
for each $n$, may arbitrarily take values from some set $\setA$.
System~\eqref{E:main} is called (asymptotically) \emph{stable} if, for each
sequence of matrices $A(n)\in\setA$, $n=0,1,\ldots$, the corresponding
solution $x(n)$ tends to zero. The asymptotic stability of switching
system~\eqref{E:main} is equivalent to the exponential convergence to zero of
each sequence $\{X(n)\}$ of the matrix products $X(n)=A(n)\cdots A(1)A(0)$
\cite{KKKK:DAN84:e,Bar:ARC88,Koz:AiT90:10:e,Gurv:LAA95,Koz:ICDEA04,SWMWK:SIAMREV07,LinAnt:IEEETAC09,ForValt:IEEETAC12},
which in turn is equivalent to the inequality
\begin{equation}\label{E-GSRle1}
\rho(\setA)<1.
\end{equation}
Here, the quantity $\rho(\setA)$, called~\cite{RotaStr:IM60} the \emph{joint
spectral radius} of the matrix set $\setA$, is defined as follows:
\begin{equation}\label{E-GSRad0}
\rho(\setA)=
\lim_{n\to\infty}\sup\left\{\|A_{n}\cdots A_{1}\|^{1/n}:~A_{i}\in\setA\right\},
\end{equation}
where $\|\cdot\|$ is an arbitrary norm on $\mathbb{R}^{N\times N}$.

For switching systems that are not stable, one may pose the question about
the existence of at least one sequence of matrices $A(n)\in\setA$,
$n=0,1,\ldots$, such that $A(n)\cdots A(1)A(0)\to 0$, that is, about
\emph{stabilization} of a system. It is
known~\cite{Gurv:LAA95,Theys:PhD05,Jungers:09,ShenHu:SIAMJCO12,BochiMor:PLMS15}
that system~\eqref{E:main} can be stabilized if the following inequality
holds:
\begin{equation}\label{E-LSRle1}
\check{\rho}(\setA)<1,
\end{equation}
where the quantity $\check{\rho}(\setA)$, called the \emph{lower spectral
radius} \cite{Gurv:LAA95} of the matrix set $\setA$, is as follows:
\begin{equation}\label{E-LSRad0}
\check{\rho}(\setA)=
\lim_{n\to\infty}\inf\left\{\|A_{n}\cdots A_{1}\|^{1/n}:~A_{i}\in\setA\right\}.
\end{equation}

Inequalities \eqref{E-GSRle1} and \eqref{E-LSRle1} might seem to give an
exhaustive answer to the questions on stability or stabilizability of a
switching system. This is indeed the case from the theoretical point of view.
However, in practice it is rather difficult, if at all possible, to calculate
in a closed formula form the limits in~\eqref{E-GSRad0} and~\eqref{E-LSRad0},
see, e.g., numerous negative results
in~\cite{BM:JAMS02,BTV:MTNS02,Koz:CDC05:e,CJ:IJAMCS07,Koz:AiT90:6:e,TB:MCSS97-1}.
This implies the need to make use of approximate computational methods.
Besides, currently there are no a priory estimates for the rate of
convergence of the limits~\eqref{E-GSRad0} and \eqref{E-LSRad0}, and the
required amount of computations rapidly increases in $n$ and the dimension of
the system, which exacerbates the difficulty in the usage of computational
methods. In this regard, we would like to note the following problems of
stability and stabilizability of linear switching systems, which are not new
per se, but remain to be relevant.

In this regard, we would like to note the following problems of stability and
stabilizability of linear switching systems, which are not new per se, but
remain to be relevant.

\begin{problem}\label{P:1}
How to describe the classes of switching systems (or equivalently, the
classes of matrix sets $\setA$), for which the joint spectral
radius~\eqref{E-GSRad0} could be constructively calculated?
\end{problem}

\begin{problem}\label{P:2}
How to describe the classes of switching systems (or equivalently, the
classes of matrix sets $\setA$), for which the lower spectral
radius~\eqref{E-LSRad0} could be constructively calculated?
\end{problem}

There is another circumstance that hampers the investigation of stability and
stabilizability of system \eqref{E:main}. This circumstance is barely
mentioned in the theory of convergence of matrix products but is of crucial
importance in control theory. The point is that, in control theory,
systems in general are composed not of a single block but of a number of
interconnected blocks. When these blocks are linear and functioning
asynchronously each of them is described by the equation
\begin{equation}\label{E:main-block}
x_{\text{out}}(n+1)=A_{i}(n)x_{\text{in}}(n),
\end{equation}
where $x_{\text{in}}(\cdot)\in\mathbb{R}^{N_{i}}$,
$x_{\text{out}}(\cdot)\in\mathbb{R}^{M_{i}}$, and the matrices $A_{i}(n)$,
for each $n$, may arbitrarily take values from some set $\setA_{i}$ of
$(N_{i}\times M_{i})$-matrices, where $i=1,2,\ldots,Q$ and $Q$ is the total
amount of blocks in the system.

\begin{figure}[htbp]
\centering
\includegraphics{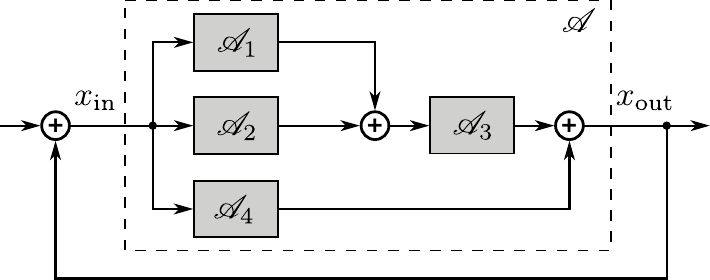}
\caption{An example of a series-parallel connection
of controllers of a system}\label{F-1}
\end{figure}

In this case it is natural to pose the question about stability or
stabilizability not for isolated blocks or controllers \eqref{E:main-block},
but for the system as a whole, whose blocks may be connected in parallel or
in series, or in a more complicated way, represented by some directed graph
with blocks of the form \eqref{E:main-block} placed on its edges, see
Fig.~\ref{F-1}. Unfortunately, under such a connection of blocks, the classes
of matrices describing the transient processes of a system as a whole became
very complicated and their properties are practically not investigated. As a
rule, even in the cases when the dimensions of the input-output vectors
coincide with each other and hence the question about stability or
stabilizability of a single block may be somehow answered, after a
series-parallel connection of such blocks, it is often impossible to
constructively resolve the question about the stability of the whole system
or, at the best, it is very difficult to get the desired answer. So, the
following problem is also urgent:

\begin{problem}\label{P:3}
How to describe the classes of switching systems for which the question about
stability or stabilizability can be constructively answered not only for an
isolated switching block \eqref{E:main} or \eqref{E:main-block} but also for
any series-parallel connection of such blocks?
\end{problem}

At last, let us consider one more aspect of the problem of constructive
stability or stabilizability of the switching systems.

The joint spectral radius~\eqref{E-GSRad0}, as well as the lower spectral
radius~\eqref{E-LSRad0}, provide only characterization of stability or
stabilizability of a system `as a whole'. They describe the limiting behavior
of the `multiplicatively averaged' norms of the matrix products,
$\|A(n-1)\cdots A(0)\|^{1/n}$. If one is interested in the study of stability
of a system, in typical situations, e.g. for the so-called \emph{irreducible}\footnote{A set of matrices is called irreducible if all the
matrices from this set do not have common invariant subspaces except the
trivial zero space and the whole space.}
classes of matrices $\setA$, \emph{for each}
sequence of matrices $\{A(n)\}$ the following estimate holds
\[
\|A(n-1)\cdots A(0)\|\le C\rho^{n}(\setA),
\]
see, e.g., \cite{Bar:ARC88}. In the case when one is interested in the study of
stabilizability of a system, in typical situations \emph{there exists} a
sequence of matrices $\{A(n)\}$ such that the following
estimate is valid:
\[
\|A(n-1)\cdots A(0)\|\le \check{C}\check{\rho}^{n}(\setA).
\]

At the same time there is often a need to find a sequence of matrices that
would ensure the slowest or fastest `decrease' not of the norms of matrix
products $\|A(n-1)\cdots A(0)\|$ but, for a given initial vector $x$, of the
vectors $A(n-1)\cdots A(0)x$. More precisely, let us consider a real function
$\nu(x)\equiv \nu(x_{1},x_{2},\ldots,x_{N})$ which is non-decreasing in each
coordinate $x_{i}$ of the vector $x=(x_{1},x_{2},\ldots,x_{N})$ and defined
for all $x_{1},x_{2},\ldots,x_{N}\ge0$. Such a function will be called
\emph{coordinate-wise monotone}, while in the case when it is strictly
increasing in each variable $x_{i}$ it will be called \emph{strictly
coordinate-wise monotone}. For example, each of the norms
\[
\|x\|_{1}=\sum_{i}|x_{i}|,\quad \|x\|_{2}=\sqrt{\sum_{i}|x_{i}|^{2}},\quad
\|x\|_{\infty}=\max_{i}|x_{i}|,
\]
is a coordinate-wise monotone function. Moreover, the norms $\|x\|_{1}$ and
$\|x\|_{2}$ are strictly coordinate-wise monotone whereas the norm
$\|x\|_{\infty}$ is coordinate-wise monotone but not strictly coordinate-wise
monotone.

If a set of matrices $\setA$ is finite and consists of $K$ elements then to
find the value of
\[
\max_{A\in\setA}\nu(Ax)
\]
it is needed, in general, to compute $K$ times the values of the function
$\nu(\cdot)$, and then to find their maximum. Similarly, to find the value of
\begin{equation}\label{E:nun}
\max_{A_{i_{j}}\in\setA}\nu(A_{i_{n}}\cdots A_{i_{1}}x)
\end{equation}
one need, in general, to compute $K^{n}$ times the values of the function
$\nu(\cdot)$, and then to find their maximum, which leads to an exponential
in $n$ growth of the number of required computations. Therefore, it is
reasonable to put the following problem:

\begin{problem}\label{P:4}
Given a coordinate-wise monotone function $\nu(\cdot)$ and a vector $x\neq0$.
How to describe the classes of switching systems (or equivalently, the
classes of matrix sets $\setA$), for which the number of computations of the
function $\nu(\cdot)$ needed to find the quantity \eqref{E:nun} would be less
than $K^{n}$? It is desirable that the required number of computations would
be of order $Kn$.
\end{problem}

Clearly, a similar problem about minimization of the quantity
$\nu(A_{i_{n}}\cdots A_{i_{1}}x)$ can also be posed.

In connection with this, our aim is to describe a class of asynchronous
blocks or controllers~\eqref{E:main}, rather simple and natural in
applications, for which one can obtain affordable answers to
Problems~\ref{P:1}--\ref{P:4}.

In Section~\ref{S:IRU}, we recall some facts from the theory of matrix
products.

\section{Sets of matrices with constructively computable spectral
characteristics}\label{S:IRU}

One of classes of matrix sets whose characteristics \eqref{E-GSRad0} and
\eqref{E-LSRad0} may be constructively calculated is the so-called class of
positive matrix sets with independent row
uncertainty~\cite{BN:SIAMJMAA09}. Recall the related definitions.

In accordance with~\cite{BN:SIAMJMAA09}, a set of $N\times M$-matrices
$\setA$ is called a \emph{set with independent row uncertainty}, or an
\emph{IRU-set}, if it consists of all the matrices
\[
A=\begin{pmatrix}
a_{11}&a_{12}&\cdots&a_{1M}\\
a_{21}&a_{22}&\cdots&a_{2M}\\
\cdots&\cdots&\cdots&\cdots\\
a_{N1}&a_{N2}&\cdots&a_{NM}
\end{pmatrix},
\]
each row $a_{i} = (a_{i1}, a_{i2}, \ldots, a_{iM})$ of which belongs to some
set of rows $\setA^{(i)}$, $i=1,2,\ldots,N$. An IRU-set of matrices will be
referred to as positive if all its matrices are positive, which is equivalent
to the positivity of all strings composing the sets $\setA^{(i)}$. The
totality of all IRU-sets of positive $(N\times M)$-matrices will be denoted
by $\IRU(N,M)$.
\begin{example}\label{Ex:0}\rm Let the sets of rows $\setA^{(1)}$ and $\setA^{(2)}$
be as follows:
\[
\setA^{(1)}=\{(a,b),~(c,d)\},~
\setA^{(2)}=\{(\alpha,\beta),~(\gamma,\delta),~(\mu,\nu)\}.
\]
Then the IRU-set $\setA$ consists of the following matrices:
\begin{alignat*}{3}
A_{11}&=\begin{pmatrix}a&b\\\alpha&\beta \end{pmatrix},&~
A_{12}&=\begin{pmatrix}a&b\\\gamma&\delta \end{pmatrix},&~
A_{13}&=\begin{pmatrix}a&b\\\mu&\nu \end{pmatrix},\\
A_{21}&=\begin{pmatrix}c&d\\\alpha&\beta \end{pmatrix},&~
A_{22}&=\begin{pmatrix}c&d\\\gamma&\delta \end{pmatrix},&~
A_{23}&=\begin{pmatrix}c&d\\\mu&\nu \end{pmatrix}.
\end{alignat*}
\end{example}

If a set $\setA$ is compact, which is equivalent to the compactness of each
set of rows $\setA^{(1)}$, $\setA^{(2)}$, \ldots, $\setA^{(N)}$, then the
following quantities are well defined:
\[
\rho_{min}(\setA) = \min_{A \in \setA} \rho(A), \quad
\rho_{max}(\setA) = \max_{A \in \setA} \rho(A).
\]
As is shown in~\cite{NesPro:SIAMJMAA13,Koz:LAA16},
\begin{equation}\label{E:finrel} \rho(\setA)=\rho_{max}(\setA),\quad
\check{\rho}(\setA)=\rho_{min}(\setA),
\end{equation}
for positive compact IRU-sets of matrices $\setA$, whereas for arbitrary sets
of matrices the equalities in~\eqref{E:finrel} are not valid,
see~\cite[Example~1]{Koz:LAA16}.

For finite IRU-sets of matrices $\setA$, the quantities $\rho_{min}(\setA)$
and $\rho_{max}(\setA)$ can be constructively calculated, and therefore due
to \eqref{E:finrel}, for finite IRU-sets of positive matrices, the quantities
$\rho(\setA)$ and $\check{\rho}(\setA)$ are also can be constructively
calculated. An efficient computational algorithm for finding the quantities
$\rho_{min}(\setA)$ and $\rho_{max}(\setA)$, for various IRU-sets of matrices
$\setA$, is proposed in \cite{Prot:MP15}.

Another example of classes of matrices, for which the quantities
\eqref{E-GSRad0} and \eqref{E-LSRad0} can be constructively calculated, is
given by the so-called \emph{linearly ordered} sets of positive matrices
$\setA=\{A_{1},A_{2},\ldots,A_{n}\}$, that is, such sets of matrices for
which $0< A_{1}< A_{2}<\cdots< A_{n}$, where the inequalities are meant
element-wise. For this class of matrices, equalities~\eqref{E:finrel} follow
from the known relations between the spectral radii of comparable positive
matrices~\cite[Corollary 8.1.19]{HJ2:e}. The totality of all linearly ordered
sets of $(N\times M)$-matrices will be denoted by $\Lin(N,M)$.

It should be noted that controllers or blocks whose behavior is covered by
equations~\eqref{E:main} or \eqref{E:main-block} with IRU-sets of matrices
are rather common asynchronous controllers in control theory which perform
the so-called \emph{independent coordinate-wise correction of the state
vectors}. The controllers whose whose behavior is covered by
equations~\eqref{E:main} or \eqref{E:main-block} with linearly ordered sets
of matrices are a kind of \emph{amplifiers with `matrix' coefficients of
amplification} varying in time.

In~\cite{Koz:LAA16} it was observed that the proofs of
equalities~\eqref{E:finrel} for the IRU-sets of positive matrices, as well as
for the linearly ordered sets of positive matrices, may be obtained by the
same scheme, as a corollary of some general principle, which we now describe
in more detail.

\subsection{Hourglass alternative}

For vectors $x,y\in\bbR^{N}$, we write $x \ge y$ ($x>y$), if all coordinates
of the vector $x$ are not less (strictly greater), than the corresponding
coordinates of the vector $y$. Similar notation will be applied to matrices.

A set of positive matrices $\setA$ is called an $\Hset$-set~\cite{Koz:LAA16}
if, for any matrix $\tilde{A}\in\setA$ and any vector $x>0$, the following
assertions hold:
\begin{quote}\em
\begin{enumerate}[\rm H1:]
\item either $Ax\ge \tilde{A}x$ for all $A\in\setA$ or there exists a matrix
    $\bar{A}\in\setA$ such that $\bar{A}x\le \tilde{A}x$ and $\bar{A}x\neq \tilde{A}x$;

\item either $Ax\le \tilde{A}x$ for all $A\in\setA$ or there exists a matrix
    $\bar{A}\in\setA$ such that  $\bar{A}x\ge \tilde{A}x$ and $\bar{A}x\neq \tilde{A}x$.
\end{enumerate}
\end{quote}

\sloppy Assertions H1 and H2 have a simple geometrical interpretation.
Imagine that the sets \mbox{$\{u:u\le \tilde{A}x\}$} and \mbox{$\{u:u\ge
\tilde{A}x\}$} form the lower and upper bulbs of some stylized hourglass with
the neck at the point $\tilde{A}x$. Then, according to Assertions H1 and H2,
either all the `grains' $Ax$ fill one of the bulbs (upper or lower), or at
least one grain remains in the other bulb (lower or upper, respectively). In
\cite{Koz:LAA16}, such an interpretation gave reason to call Assertions H1
and H2 the \emph{hourglass alternative}.

\fussy The totality of all compact\footnote{The set of all $(N\times
M)$-matrices is naturally endowed by the topology of element-wise convergence
which allows do define the concept of compactness for the related sets of
matrices.} $\Hset$-sets of matrices of dimension $N\times M$ will be denoted
by $\Hset(N,M)$. Then the main result about the spectral properties of the
$\Hset$-sets of matrices can be formulated as follows.

\begin{theorem}[see \cite{Koz:LAA16}]\label{T:Hset}
Let $\setA\in\Hset(N,N)$. Then equalities~\eqref{E:finrel} hold.
\end{theorem}

As a matter of fact, in \cite{Koz:LAA16} a number of more
profound results are proved, but we will not delve into the intricacies.

\subsection{$\Hset$-sets of matrices}\label{S:Hsets}

The applicability of Theorem~\ref{T:Hset} essentially depends on how
constructive one will be able to describe the classes of $\Hset$-sets of
matrices. In~\cite{Koz:LAA16} it was shown that the sets of matrices with
independent row uncertainty and the linearly ordered sets of positive
matrices are $\Hset$-sets of matrices. However, as demonstrates
Example~\ref{Ex:nonH} below, not every set of positive matrices is an
$\Hset$-set. The one-element sets of matrices $\{0\}$ and $\{I\}$ consisting
of the zero and the identity matrices are also not $\Hset$-sets because the
related matrices are not positive.

\begin{example}\label{Ex:nonH}\rm
Let us consider the set of matrices $\setA=\{A_{1},A_{2}\}$, where
\[
A_{1}=\begin{pmatrix}a&a^{2}\\1&a\hphantom{~}\end{pmatrix},\quad
A_{2}=\begin{pmatrix}a\hphantom{~}&1\\a^{2}&a\end{pmatrix},\qquad a>0.
\]
Then $\max\{\rho(A_{1}),\rho(A_{2})\}=2a$ and
$\rho(A_{1}A_{2})=(1+a^{2})^{2}$. Therefore, for $a\neq 1$,
\[
\rho(\setA)\ge \|A_{1}A_{2}\|^{1/2}\ge\rho(A_{1}A_{2})^{1/2}>\max\{\rho(A_{1}),\rho(A_{2})\},
\]
which by Theorem~\ref{T:Hset} could not be valid if $\setA$ was an
$\Hset$-set of matrices.
\end{example}

To construct other classes of $\Hset$-sets of matrices let us ascertain some
general properties of such sets of matrices. Introduce the operations of
Minkowski addition and multiplication for sets of matrices:
\begin{align*}
\setA+\setB &:=\{A+B:A\in\setA,~ B\in\setB\},\\
\setA\setB &:=\{AB:A\in\setA,~ B\in\setB\},
\end{align*}
and also the operation of multiplication of a set of matrices by a number:
\[
t\setA=\setA t:=\{tA:t\in\bbR,~A\in\setA\}.
\]
The Minkowski addition of sets of matrices corresponds to the parallel
coupling of two independently operating asynchronous controllers, while the
Minkowski multiplication corresponds to the serial connection of such
asynchronous controllers.

\begin{remark}\label{R:nonass}\rm
In general, $\setA(\setB_{1}+\setB_{2})\neq \setA \setB_{1} +\setA\setB_{2}$
and $(\setA_{1}+\setA_{2})\setB\neq \setA_{1}\setB +\setA_{2}\setB$, i.e. the
Minkowski operations are not associative. In particular, $\setA+\setA \neq
2\setA$.
\end{remark}

Clearly, the operation of addition is \emph{admissible} if the matrices from
the set $\setA$ are of the same size as the matrices from the set $\setB$ ,
while the operation of multiplication is \emph{admissible} if the sizes of
the matrices from sets $\setA$ and $\setB$ are matched: the dimension of the
rows of the matrices from $\setA$ is the same as the dimension of the columns
of the matrices from $\setB$. There is no problem with matching of sizes when
one considers sets of square matrices of the same size.

\begin{theorem}[see \cite{Koz:LAA16}]\label{T:semiring} The following is true:
\begin{enumerate}[\rm(i)]
  \item $\setA+\setB\in\Hset(N,M)$, if $\setA,\setB\in\Hset(N,M)$;
  \item $\setA\setB\in\Hset(N,Q)$, if $\setA\in\Hset(N,M)$ and
      $\setB\in\Hset(M,Q)$;
  \item $t\setA=\setA t\in\Hset(N,M)$, if $t>0$ and $\setA\in\Hset(N,M)$.
\end{enumerate}
\end{theorem}

By Theorem~\ref{T:semiring} the totality of sets of square matrices
$\Hset(N,N)$ is endowed with additive and multiplicative binary operations,
but itself is not a group, neither additive nor multiplicative. However,
after adding the zero additive element $\{0\}$ and the identity
multiplicative element $\{I\}$ to $\Hset(N,N)$, the resulting totality
$\Hset(N,N)\cup\{0\}\cup\{I\}$ becomes a semiring~\cite{Golan99}.

The fact that the totality $\Hset(N,N)$ is endowed with the operations of
addition and multiplication means that, by connecting in a serial-parallel
manner independently operating asynchronous controllers that satisfy the
axioms H1 and H2, we again obtain an asynchronous controller satisfying the
axioms H1 and H2.

\begin{remark}\label{R:1}\rm
Theorem~\ref{T:semiring} implies that any finite sum of any finite products
of sets of matrices from $\Hset(N,N)$ is again a set of matrices from
$\Hset(N,N)$. Moreover, for any integers $n,d\ge1$, all the polynomial sets
of matrices
\begin{equation}\label{E:poly}
P(\setA_{1},\setA_{2},\ldots,\setA_{n})=
\sum_{k=1}^{d}\sum_{i_{1},i_{2},\ldots,i_{k}\in\{1,2,\ldots,n\}}
p_{i_{1},i_{2},\ldots,i_{k}}\setA_{i_{1}}\setA_{i_{2}}\cdots\setA_{i_{k}},
\end{equation}
where $\setA_{1},\setA_{2},\ldots,\setA_{n}\in\Hset(N,N)$ and the scalar
coefficients $p_{i_{1},i_{2},\ldots,i_{k}}$ are positive, belong to the set
$\Hset(N,N)$.

With the help of polynomials \eqref{E:poly} one can construct not only the
elements of the set $\Hset(N,N)$ but also the elements of arbitrary sets
$\Hset(N,M)$, by taking the arguments $\setA_{1},\setA_{2},\ldots,\setA_{n}$
from the sets $\Hset(N_{i},M_{i})$ with arbitrary matrix sizes $N_{i}\times
M_{i}$. One must only ensure that the products
$\setA_{i_{1}}\setA_{i_{2}}\cdots\setA_{i_{k}}$ were admissible, and the
expression \eqref{E:poly} would determine the sets of matrices of dimension
$N\times M$.

We have presented above two types of non-trivial $\Hset$-sets of matrices,
the sets of positive matrices with independent row uncertainty and the
linearly ordered sets of positive matrices. In this connection, let us denote
by $\Hset_{*}(N,M)$ the totality of all sets of $(N\times M)$-matrices which
can be obtained as the recursive expansion with the help of
polynomials~\eqref{E:poly} of the sets of positive matrices with independent
rows uncertainty and the sets of linearly ordered positive matrices. In other
words, $\Hset_{*}(N,M)$ is the totality of all sets of matrices that can be
represented as the values of superpositions of matrix polynomials
\eqref{E:poly} with the arguments of the polynomials of the `lowest level'
taken from the sets of the matrices belonging to
$\IRU(N_{i},M_{i})\cup\Lin(N_{i},M_{i})$.

As was noted in Remark~\ref{R:nonass} the Minkowski operations are not
associative. Therefore the  recursive extension of the set of positive
matrices with independent rows uncertainty and of linearly ordered positive
matrices forms a wider variety of matrices than the extension of the set of
positive matrices with independent rows uncertainty and of linearly ordered
positive matrices with the help of polynomials~\eqref{E:poly}.
\end{remark}

\section{Main result}

Theorems~\ref{T:Hset} and \ref{T:semiring}, and Remark~\ref{R:1} imply the
following statement:

\begin{theorem}\label{T:main}
Given a system \eqref{E:main} formed by a series-parallel recursive
connection of blocks \eqref{E:main-block} (i.e. represented by some graph
obtained by applying recursively series and/or parallel extensions starting
form one edge, and with blocks placed on its edges) corresponding to some
$\Hset$-sets of positive matrices $\setA_{i}$, $i=1,2,\ldots,Q$. Then the
question of the stability (stabilizability) of such a system can be
constructively resolved by finding a matrix that maximizes (minimizes) the
quantity $\rho(A)$ over the set of matrices $\setA$, where $\setA$ is the
Minkowski polynomial sum \eqref{E:poly} of the sets of matrices $\setA_{i}$,
$i=1,2,\ldots,Q$, corresponding to the structure of coupling of the related
blocks.
\end{theorem}

\begin{example}\label{Ex:poly}\rm
For the system $\setA$ in Fig.~\ref{F-1}, the input and output are related by
the equality
\[
x_{\text{out}}(n+1)=\bigl(A_{3}(n)(A_{1}(n)+A_{2}(n))+A_{4}(n)\bigr)x_{\text{in}}(n),
\]
where, for each $n$, the matrices $A_{1}(n),A_{2}(n),A_{3}(n)$ and $A_{4}(n)$
are randomly selected from the related sets: $A_{i}(n)\in\setA_{i}$,
$i=1,2,3,4$. Correspondingly, in this case all the possible values of the
transition matrix for the system $\setA$ can be obtained as the elements of
the following Minkowski polynomial sum of the sets of matrices
$\setA_{1},\setA_{2},\setA_{3},\setA_{4}$:
\[
P(\setA_{1},\setA_{2},\setA_{3},\setA_{4})=
\setA_{3}(\setA_{1}+\setA_{2})+\setA_{4}.
\]
\end{example}

\section{Construction of individual maximizing and minimizing sequences}

\subsection{One-step maximization}

We first consider the problem of maximizing the function $\nu(Ax)$, where
$x>0$, over all $A$ from the $\Hset$-set $\setA$, which is assumed to be
compact. By Assertion~H2 of the hourglass alternative, for any matrix
$\tilde{A}\in\setA$, either $Ax\le\tilde{A}x$ for all $A\in\setA$ or there
exists a matrix $\bar{A}\in\setA$ such that $\bar{A}x\ge\tilde{A}x$ and
$\bar{A}x\neq\tilde{A}x$. This together with the compactness of the set
$\setA$ implies the existence of a matrix $A^{(max)}\in\setA$ such that, for
all $A\in\setA$, the following inequality holds:
\begin{equation}\label{E:Amaxineq}
  Ax\le A^{(max)}x.
\end{equation}
Let us notice that the matrix $A^{(max)}$ depends on the vector $x$, and
therefore, when needed, we will write $A^{(max)}=A^{(max)}_{x}$. Moreover,
the matrix $A^{(max)}_{x}$ is generally determined non-uniquely by the vector
$x$.

\begin{theorem}\label{T:1stepmax}
Let $\setA$ be a compact $\Hset$-set of positive $(N\times N)$-matrices,
$\nu(\cdot)$ be a coordinate-wise monotone function, and
$x\in\mathbb{R}^{N}$, $x>0$, be a vector.

\begin{enumerate}[\rm(i)]
\item Then the maximum of the
    function $\nu(Ax)$ over $A\in\setA$ is  attained at the matrix $A^{(max)}=A^{(max)}_{x}$, that is,
\[
\max_{A\in\setA}\nu(Ax)= \nu(A^{(max)}x).
\]
\item If the maximum of the function $\nu(Ax)$ over $A\in\setA$ is attained
    at a matrix $A_{0}\in\setA$ and the function $\nu(\cdot)$ is strictly
    coordinate-wise monotone, then $A_{0}x=A^{(max)}_{x}x$.
\end{enumerate}
\end{theorem}

\begin{proof}
Assertion~(i) directly follows from inequality \eqref{E:Amaxineq} and the
coordinate-wise monotonicity of the function $\nu(\cdot)$.

To prove Assertion~(ii) let us notice that
\[
A_{0}x\le A^{(max)}_{x}x.
\]
If here $A_{0}x\neq A^{(max)}_{x}x$ then at least one coordinate of the
vector $A^{(max)}_{x}x$ should be strictly greater than the respective
coordinate of the vector $A_{0}x$. Then, due to the strict coordinate-wise
monotonicity of the function $\nu(\cdot)$, the following inequality holds:
\[
\nu(A_{0}x)<\nu(A^{(max)}_{x}x),
\]
which contradicts to the assumption that the maximum of the function
$\nu(Ax)$ over $A\in\setA$ is attained at the matrix $A_{0}\in\setA$.
Therefore, $A_{0}x=A^{(max)}_{x}x$, and Assertion~(ii) is proved.
\end{proof}

\begin{remark}\label{Rem:ii}\rm
If the function $\nu(\cdot)$ is coordinate-wise monotone but not strictly
coordinate-wise monotone then, in general, Assertion~(ii) of
Theorem~\ref{T:1stepmax} is not valid.
\end{remark}

\begin{remark}\label{Rem:1}\rm
The construction of the matrix $A^{(max)}$ does not depend on the function
$\nu(\cdot)$.
\end{remark}

\subsection{Multi-step maximization: solution of Problem~\ref{P:4}}

We turn now to the question of determining the quantity \eqref{E:nun} for
some $n>1$ and $x\in\mathbb{R}^{N}$, $x> 0$. With this aim in view, let us
construct sequentially the matrices $A^{(max)}_{i}$, $i=1,2,\ldots,n$, as
follows:

\begin{itemize}
\item the matrix $A^{(max)}_{1}$, depending in the vector $x_{0}=x$, is
    constructed in the same way as was done in the previous section:
    $A^{(max)}_{1}=A^{(max)}_{x_{0}}$;

\item if the matrices $A^{(max)}_{i}$, $i=1,2,\ldots,k$, have already
    constructed then the matrix $A^{(max)}_{k+1}$, depending on the vector
\[
x_{k}=A^{(max)}_{k}\cdots A^{(max)}_{1}x,
\]
is constructed to maximize the function
\[
\nu(AA^{(max)}_{k}\cdots
A^{(max)}_{1}x)=\nu(Ax_{k})
\]
over all $A\in\setA$ in the same manner as was done in the previous
section. So, the matrix $A^{(max)}_{k+1}$ is defined by the equality
$A^{(max)}_{k+1}=A^{(max)}_{x_{k}}$.
\end{itemize}

By definition of the matrices $A^{(max)}_{i}$ then, in view of
\eqref{E:Amaxineq}, for all $A\in\setA$ the following inequalities hold:
\begin{align}
\notag  Ax&\le A^{(max)}_{1}x,\\
\notag    AA^{(max)}_{1}x&\le A^{(max)}_{2}A^{(max)}_{1}x,\\
\notag    &\ldots\\
\notag  AA^{(max)}_{n-1}\cdots A^{(max)}_{1}x&\le A^{(max)}_{n}\cdots A^{(max)}_{1}x,
\intertext{which implies}
\label{E:Amaxineq-n}
A_{n}\cdots A_{1}x&\le A^{(max)}_{n}\cdots A^{(max)}_{1}x
\end{align}
for all $A_{n},\ldots,A_{1}\in\setA$.

\begin{theorem}\label{T:nstepmax}
Let $\setA$ be a compact $\Hset$-set of positive $(N\times N)$-matrices,
$\nu(\cdot)$ be a coordinate-wise monotone function, and
$x\in\mathbb{R}^{N}$, $x>0$, be a vector.

\begin{enumerate}[\rm(i)]
\item Then the maximum of the function $\nu(A_{n}\cdots A_{1}x)$ over
    $A_{1},\ldots,A_{n}\in\setA$ is attained at the set of matrices
    $A^{(max)}_{1},\ldots, A^{(max)}_{n}$, that is,
\[
\max_{A_{n},\ldots,A_{1}\in\setA}\nu(A_{n}\cdots A_{1}x)=
\nu(A^{(max)}_{n}\cdots A^{(max)}_{1}x).
\]

\item Let $\setA$ be a compact $\Hset$-set of positive matrices. If the
    maximum of the function $\nu(A_{n}\cdots A_{1}x)$ over
    $A_{n},\ldots,A_{1}\in\setA$ is attained at a set of matrices
    $\tilde{A}_{1},\ldots, \tilde{A}_{n}$ and the function $\nu(\cdot)$ is
    strictly coordinate-wise monotone, then
    \begin{equation}\label{E:ieqs}
    \tilde{A}_{i}\cdots\tilde{A}_{1}x=A^{(max)}_{i}\cdots A^{(max)}_{1}x,\qquad i=1,2,\ldots,n.
    \end{equation}
\end{enumerate}
\end{theorem}

\begin{proof}
Assertion~(i) directly follows from inequality \eqref{E:Amaxineq-n} and the
coordinate-wise monotonicity of the function $\nu(\cdot)$.

To prove Assertion~(ii) let us observe that
\begin{align*}
  \tilde{A}_{1}x&\le A^{(max)}_{1}x,\\
  \tilde{A}_{2}\tilde{A}_{1}x&\le A^{(max)}_{2}A^{(max)}_{1}x,\\
  &\ldots\\
\tilde{A}_{n}\tilde{A}_{n-1}\cdots \tilde{A}_{1}x&\le A^{(max)}_{n}\cdots A^{(max)}_{1}x,
\end{align*}
If here equalities \eqref{E:ieqs} are not valid for some $i=i_{0}$ but valid
for all $i<i_{0}$ then at least one coordinate of the vector
$A^{(max)}_{i_{0}}\cdots A^{(max)}_{1}x$ is strictly greater than the respective coordinate of
the vector $\tilde{A}_{i_{0}}\cdots \tilde{A}_{1}x$. Then, due to the positivity of the matrices
from the set $\setA$, for each $j\ge i_{0}$ there is valid the inequality
\[
\tilde{A}_{j}\tilde{A}_{j-1}\cdots \tilde{A}_{1}x\le
A^{(max)}_{j}\cdots A^{(max)}_{1}x,
\]
where at least one coordinate of the vector $A^{(max)}_{j}\cdots
A^{(max)}_{1}x$ is strictly greater\footnote{This argument `fails', if we
assume that the matrices constituting the set $\setA$ are only positive.}
than the respective coordinate of the vector
$\tilde{A}_{j}\tilde{A}_{j-1}\cdots \tilde{A}_{1}x$. Then, due to the strict
coordinate-wise monotonicity of the function $\nu(\cdot)$, for $j=n$ we
obtain the inequality
\[
\nu(\tilde{A}_{n}\cdots \tilde{A}_{1}x)<
\nu(A^{(max)}_{n}\cdots A^{(max)}_{1}x),
\]
contradicting to the assumption that the maximum of the function
$\nu(A_{n}\cdots A_{1}x)$ over $A_{n},\ldots,A_{1}\in\setA$ is attained at
the set of matrices $\tilde{A}_{1}\ldots, \tilde{A}_{n}$. Therefore,
equalities \eqref{E:ieqs} should be valid for all $i=1,2,\ldots,n$, and
Assertion~(ii) is proved.
\end{proof}

\begin{remark}\label{Rem:2}\rm
The construction of each subsequent matrix $A^{(max)}_{i}$ is `positional'
or, what is the same, it is made in accordance with the `principles of
dynamic programming', that is, only based on the information known up to this
step. At the same time, this construction does not depend on the function
$\nu(\cdot)$, and hence on the complexity of its calculation!
\end{remark}

\section{Non-negative matrices}

In the previous sections, all the considerations have been carried out for
classes of matrices with positive elements. Sometimes, the requirement of
positivity of the related matrices may be restrictive, however the transition
to the matrices with arbitrary elements is hardly possible in the context of
the treated problems, see~\cite{Koz:LAA16} and the discussion therein. Even
the transition to matrices with non-negative elements is not always possible,
since in general, for such matrices, the constructions and statements of
Section~\ref{S:IRU} are no longer valid. Nevertheless, in one particular case
of practical interest the transition to non-negative matrices is possible.

Denote by $\overline{\IRU}(N,M)$ the totality of all IRU-sets of non-negative
$(N\times M)$-matrices, and by $\overline{\Lin}(N,M)$ denote the totality of
all sets $\setA=\{A_{1},A_{2},\ldots,A_{n}\}$ of non-negative $(N\times
M)$-matrices satisfying the inequalities $0\le A_{1}\le A_{2}\le\cdots\le
A_{n}$. The totalities of sets of non-negative matrices
$\overline{\IRU}(N,M)$ and $\overline{\Lin}(N,M)$ can be naturally treated as
a kind of `closure' of the related totalities of positive matrices
$\IRU(N,M)$ and $\Lin(N,M)$.

Now, denote by $\overline{\Hset}_{*}(N,M)$ the totality of all sets of
matrices that can be represented as the values of polynomials \eqref{E:poly}
with the arguments taken from the sets of matrices belonging to
$\overline{\IRU}(N_{i},M_{i})\cup\overline{\Lin}(N_{i},M_{i})$. In this case,
the totality $\overline{\Hset}_{*}(N,M)$ is no longer belongs to $\Hset(N,M)$
but, as was shown in~\cite{Koz:LAA16}, for each matrix
$\setA\in\overline{\Hset}_{*}(N,N)$ equalities \eqref{E:finrel} remain valid,
i.e. an analog of Theorem~\ref{T:Hset} holds.

\section{Conclusion}
One of the most prominent problem in the design of control systems with
switching components is that of evaluating (computing) the joint or lower
spectral radii of the resulting system which determine its stability or
stabilizability, respectively.

The approach to resolving this problem proposed in the article is fulfilled
in compliance with the concept of modular design of control systems. It can
be compared with the creation of toys with the help of a LEGO$^{\circledR}$
kit.

Recall that any LEGO$^{\circledR}$ kit consists of pieces (bricks and plates
with stubs) arranging which in almost arbitrary order (oriented due to the
presence of stubs) one can create a variety of structures.

Each $\Hset$-set of matrices $\setA$ also can be interpreted as a kind of a
LEGO$^{\circledR}$ kit for assembling control systems whose pieces (bricks
and plates in a LEGO$^{\circledR}$ kit) are the switching blocks
(controllers) with the transition characteristics determined by the matrix
sets $\setA_{i}\in\setA$. Then, as was shown above, \emph{any}
series-parallel recursive connection of these blocks will result in creation
of a system whose joint and lower spectral radii \emph{always} can be
computed constructively by formula~\eqref{E:finrel}.

\section*{Acknowledgments}
The work was carried out at the Kotel'nikov Institute of Radio-engineering
and Electronics, Russian Academy of Sciences, and was funded by the Russian
Science Foundation, Project No. 16-11-00063.

\bibliographystyle{elsarticle-num-nourl}
\bibliography{kozbib,kozpub}

\end{document}